\documentclass[10pt]{article}
\pdfoutput=1

\usepackage{amstext}
\usepackage{amssymb}
\usepackage{amsmath}
\usepackage{amsthm}

\usepackage[svgnames]{xcolor}
\usepackage[ascii]{inputenc}
\usepackage[british]{babel}
\usepackage{bm}

\usepackage{hyperref}
\theoremstyle{plain}
\newtheorem{theorem}{Theorem}[section]

\newtheorem{corol}[theorem]{Corollary}

\theoremstyle{definition}

\newtheorem{remark}[theorem]{Remark}
\newtheorem{example}[theorem]{Example}

\begin{document}
\vspace{\baselineskip} \thispagestyle{empty}

%%%%%%%%%%%%%%%%%%%%%%%%%%%%%%%%%%%%%%%%%%%%%%%%%%%%%%%%%%%%%%%%%%%%%%%%%%%%%%%%%%%%
\title{On the spectrum of the tridiagonal matrices with two-periodic main diagonal}
%%%%%%%%%%%%%%%%%%%%%%%%%%%%%%%%%%%%%%%%%%%%%%%%%%%%%%%%%%%%%%%%%%%%%%%%%%%%%%%%%%%%

\author{
Alexander Dyachenko\thanks{Keldysh Institute of Applied Mathematics, Russian Academy of Sciences, 125047, Moscow,  Russia}
\\
\small Email: {\tt diachenko@sfedu.ru}
 \and
Mikhail Tyaglov%
%\thanks{Moscow Center for Fundamental and Applied Mathematics, Moscow, 119991, Russia}
\thanks{School of Mathematical Sciences and MOE-LSC, Shanghai Jiao Tong University, Shanghai, 200240, P.R. China}
\\
\small Email: {\tt tyaglov@mail.ru}}

\date{\small \today}

\maketitle

\abstract{We find the spectrum and eigenvectors of an arbitrary irreducible complex tridiagonal matrix with two-periodic main
diagonal provided that the spectrum and eigenvectors of the matrix with the
same sub- and superdiagonals and zero main diagonal is known. Our result substantially
generalises some recent results on the Sylvester-Kac matrix and its certain main principal submatrices.}

\vspace{5mm}


\vspace{2mm}


\vspace{5mm}

\setcounter{equation}{0}
%%%%%%%%%%%%%%%%%%%%%%%%%%%%%%%%%%%%%%%%%%%%%%%%%%%%%%%%%%%%%%%%%%%%%%
\section{Introduction}
%%%%%%%%%%%%%%%%%%%%%%%%%%%%%%%%%%%%%%%%%%%%%%%%%%%%%%%%%%%%%%%%%%%%%%

The present note is intended to simplify and generalise some recent findings on certain modifications of the so-called Sylvester-Kac matrix
\begin{equation*}%\label{Sylvester.Kac.Matrix}
K_N:=\begin{pmatrix}
  0   &   1 &  0   &\dots&   0    &   0    &   0   \\
N &   0   &  2 &\dots&   0    &   0    &   0   \\
   0   & N-1 &   0   &\dots&   0    &   0    &   0   \\
\vdots&\vdots&\vdots&\ddots&\vdots&\vdots&\vdots\\
   0   &   0   &   0   &\dots& 0 &   N-1 &0\\
   0   &   0   &   0   &\dots&2&   0    &N\\
   0   &   0   &   0   &\dots&0&1 & 0    \\
\end{pmatrix}.
\end{equation*}
This matrix seemingly appeared for the first time in a work of J.J.\;Sylvester~\cite{Sylvester} and later was studied
by a number of mathematicians in XIX century, see~\cite{Muir.II} for references. In XX century,
Mark Kac rediscovered this matrix in his famous work on the Brownian motion~\cite{Kac}. A good survey
on the Sylvester-Kac matrix was presented by O.\,Taussky and J.\,Todd in~\cite{TausskyTodd}. References to some recent results on the Sylvester-Kac matrix
are given, e.g., in~\cite{DyTy_1}. Here
we avoid reviewing this topic: instead, we show that some recent results around Sylvester-Kac-like matrices can be extended to
an interesting property of a much larger class of tridiagonal matrices.

In the paper~\cite{Kilic_2013}, the author added to $K_N$ the main diagonal where the
entries with even and odd indices have the same absolute value but opposite signs,
and found the determinant of that matrix. The authors of~\cite{KilicArikan_2016}
calculated the determinant of a matrix obtained from $K_N$ by adding a non-zero two-periodic
main diagonal. Note that a shorter proof of the result of~\cite{KilicArikan_2016} was
given in~\cite{Fonseca_2018}.
From our Section~\ref{section:two.periodic} it is clear that the results of~\cite{Fonseca_2018,KilicArikan_2016} immediately follow from the result of~\cite{Kilic_2013}.
%In Section~\ref{section:two.periodic}, we show that the results of~\cite{Fonseca_2018,KilicArikan_2016} immediately follow from the result of~\cite{Kilic_2013}.

In the paper~\cite{FonsecaKilic_2019}, the authors dealt with a certain
submatrix of the Sylvester-Kac matrix $K_N$ (considered much earlier by A.\,Caley~\cite{Caley.1858}). Among other results, they found the determinant and the eigenvalues of this submatrix with
an added two-periodic main diagonal.

In this paper, we generalise the aforementioned results of the works~\cite{Fonseca_2018,FonsecaKilic_2019,Kilic_2013,KilicArikan_2016}
to the class of arbitrary (irreducible) complex tridiagonal matrices with two-periodic main diagonal. We show that the findings
of~\cite{Kilic_2013} can be easily extended (we give an independent proof) to such a class of matrices\footnote{It is clear
that the Sylvester-Kac matrix and the matrix considered in~\cite{FonsecaKilic_2019} belong to this class,
since they are irreducible tridiagonal and have zero main diagonal.}. This result allows us to generalise~\cite{KilicArikan_2016} and (partially)~\cite{FonsecaKilic_2019}.
Namely, we determine the eigenvalues and the determinant of a tridiagonal matrix with two-periodic main diagonal via the eigenvalues of the same matrix but with zero main diagonal -- treating the former (perturbed) matrix as a two-periodic perturbation of the latter (unperturbed) matrix.

The paper is organised as follows.  Section~\ref{section:zero.main.diag} is devoted to a brief review
of spectral properties of tridiagonal matrices with zero main diagonal. In Section~\ref{section:alternate},
we express the eigenvalues, eigenvectors, and first generalised eigenvectors of the perturbed matrix
through those
of the corresponding unperturbed matrix.
In Section~\ref{section:two.periodic}, we extend the results of
Section~\ref{section:alternate} to tridiagonal matrices with two-periodic main diagonal.
%In Section~\ref{section:two.periodic}, we obtain spectrum for tridiagonal matrices with
%two-periodic main diagonal from the result of Section~\ref{section:alternate}, although Section~\ref{section:alternate} deals with a particular case of the matrix from
%Section~\ref{section:two.periodic}.

Note that irreducible tridiagonal matrices are closely related to orthogonal polynomials. In our further work~\cite{DyTy_2}
we present alternative proofs of the results of the present paper and exact formulas for eigenvectors' interrelations.

%Additionally, in Section 3 we express the eigenvectors of matrix with two-periodic main diagonal via eigenvectors of the matrix with the
%same sub- and superdiagonals and zero main diagonal.

\setcounter{equation}{0}
%%%%%%%%%%%%%%%%%%%%%%%%%%%%%%%%%%%%%%%%%%%%%%%%%%%%%%%%%%%%%%%%%%%%%%%%%%%%%%%%%%%%%
\section{Tridiagonal matrices with zero main diagonal}\label{section:zero.main.diag}
%%%%%%%%%%%%%%%%%%%%%%%%%%%%%%%%%%%%%%%%%%%%%%%%%%%%%%%%%%%%%%%%%%%%%%%%%%%%%%%%%%%%%

Consider an $n\times n$ complex irreducible tridiagonal matrix whose main diagonal only contains zero entries
\begin{equation}\label{main.matrix.intro}
J_n=
\begin{pmatrix}
    0 & c_1 &  0 &\dots&   0   & 0 \\
    a_1 & 0 &c_2 &\dots&   0   & 0 \\
     0  &a_2 & 0 &\dots&   0   & 0 \\
    \vdots&\vdots&\vdots&\ddots&\vdots&\vdots\\
     0  &  0  &  0  &\dots&0& c_{n-1}\\
     0  &  0  &  0  &\dots&a_{n-1}&0\\
\end{pmatrix},\qquad a_k,c_k\in\mathbb{C}\setminus\{0\}.
\end{equation}
Since the main diagonal of $J_n$ is zero, the spectrum of $J_n$ is symmetric w.r.t. zero.
\begin{theorem}\label{th:J.spectrum}
The spectrum of the matrix $J_n$ has the form
\begin{equation}\label{spectrum.J}
\sigma\left(J_{2l}\right)=\{\pm\lambda_1,\pm\lambda_2,\ldots,\pm\lambda_l\}\qquad\text{and}\qquad
\sigma\left(J_{2l+1}\right)=\{0,\pm\lambda_1,\pm\lambda_2,\ldots,\pm\lambda_l\}.
\end{equation}
Here the numbers $\lambda_i$ are not necessary distinct: each eigenvalue of~$J_n$ appears as many times as its algebraic multiplicity. For even~$n$ all $\lambda_i$ are non-zero.
\end{theorem}
\begin{proof}
Let $\chi_k(z)$, $k=1,\ldots,n$, be the characteristic polynomial of the $k^{th}$ leading principal submatrix of~$J_n$. Then the
following three-term recurrence relations hold
\begin{equation}\label{TTR}
\chi_{k+1}(z)=z\chi_{k}(z)-a_kc_k\chi_{k-1}(z),\quad k=0,1,\ldots,n-1,
\end{equation}
with $\chi_{-1}(z)\equiv0$, $\chi_{0}(z)\equiv1$. From~\eqref{TTR} it follows that the polynomials $\chi_k(z)$ do not depend on $a_k$ and $c_k$
separately -- only on the product~$a_k c_k$. Therefore, the matrices $J_n$ and $-J_n$ have the same eigenvalues, and hence the spectrum of $J_n$ is
symmetric w.r.t. $0$. In particular, if~$n$ is odd, the matrix $J_n$ is singular.

However,
\begin{equation*}
\det(J_{2l})=(-1)^l\prod\limits_{k=1}^{l}a_{2k-1}c_{2k-1}\neq0,
\end{equation*}
so for even $n$, the matrix $J_n$ is non-singular.
\end{proof}

At the same time, the matrix $J_n$ may have the zero eigenvalue of any odd multiplicity.

\begin{example}
When $n$ is odd, the matrix~\eqref{main.matrix.intro} can even be nilpotent. For instance, zero is the only eigenvalue of
the matrix
\begin{equation}
\begin{pmatrix}
0&1&0&0&0\\
1&0&1&0&0\\
0&1&0&-4&0\\
0&0&1&0&2\\
0&0&0&1&0\\
\end{pmatrix}
.
\end{equation}
\end{example}

It is a folklore that any irreducible\footnote{A tridiagonal matrix is irreducible whenever its sub- and superdiagonals contain no zeroes.} tridiagonal matrix $J$ is non-derogatory, that is, all its eigenvalues are geometrically simple.
Indeed, if $\lambda$ is a an eigenvalue of $J$, then $J-\lambda I_n$, where $I_n$ is
the $n\times n$ identity matrix, is singular but its submatrix obtained from $J-\lambda I_n$ by deleting, say,
the first column and the last row, is regular. Thus, every eigenvalue has only one eigenvector.

From the form of the matrix~\eqref{main.matrix.intro} and its spectrum, it follows that
the eigenvectors and the generalised eigenvectors of the eigenvalues $\lambda_i$ and $-\lambda_i$ are closely related.

\begin{theorem}\label{th:lam-lam}
Let $\lambda$ and $-\lambda$ be eigenvalues of the matrix $J_n$ of multiplicity $k\geqslant1$. If
$\textbf{u}_0^{(\lambda)}$ is the eigenvector and $\textbf{u}_j^{(\lambda)}$, $j=1,\ldots,k-1$, are the generalised eigenvectors
of $J_n$ corresponding to $\lambda$, then
\begin{equation*}
\textbf{u}_0^{(-\lambda)}=E_n\textbf{u}_0^{(\lambda)},
\end{equation*}
\begin{equation}\label{vectors.v}
\textbf{u}_j^{(-\lambda)}=(-1)^{j}E_n\textbf{u}_j^{(\lambda)},\quad i=1,\ldots,k-1,
\end{equation}
are the eigenvector and the generalised eigenvectors of $J_n$ corresponding to $-\lambda$.
Here the matrix $E_n=\{e_{ij}\}_{i,j=1}^{n}$ is defined as follows
\begin{equation}\label{Matrix.E}
e_{ij}=
\begin{cases}
&(-1)^{i-1}\quad\text{if}\quad i=j,\\
&\quad0\quad\qquad\text{if}\quad i\neq j.\\
\end{cases}
\end{equation}
\end{theorem}

\begin{remark}\label{Matrix.E-lam0}
Substitution of~$\lambda=0$ shows that the eigenvector~$\textbf{u}_0^{(0)}$ only has zero odd components; more generally, for~$j=0,1,\dots$ all odd components of~$\textbf{u}_{2j}^{(0)}$ and even components of~$\textbf{u}_{2j+1}^{(0)}$ are equal to zero.
\end{remark}
\begin{proof}[Proof of Theorem~\ref{th:lam-lam}]
Indeed, it is easy to see that
\begin{equation*}
J_n-\lambda I_n=-E_n(J_n+\lambda I_n)E_n,
\end{equation*}
where $I_n$ is the $n\times n$ identity matrix. By definition,
\begin{equation*}
(J_n-\lambda I_n)\textit{\textbf{u}}_0^{(\lambda)}=0,\qquad (J_n-\lambda I_n)\textit{\textbf{u}}_j^{(\lambda)}=\textit{\textbf{u}}_{j-1}^{(\lambda)},\ \ \ j=1,\ldots,k-1,
\end{equation*}
so we have
\begin{equation*}
-E_n(J_n+\lambda I_n)E_n\textit{\textbf{u}}_0^{(\lambda)}=0,\qquad -E_n(J_n+\lambda I_n)E_n\textit{\textbf{u}}_j^{(\lambda)}=\textit{\textbf{u}}_{j-1}^{(\lambda)},\ \ \ j=1,\ldots,k-1,
\end{equation*}
or
\begin{equation*}
(J_n+\lambda I_n)\textit{\textbf{u}}_0^{(-\lambda)}=0,\qquad (J_n+\lambda I_n)\textit{\textbf{u}}_j^{(-\lambda)}=\textit{\textbf{u}}_{j-1}^{(-\lambda)},\ \ \ j=1,\ldots,k-1,
\end{equation*}
where $\textit{\textbf{u}}_j^{(-\lambda)}$ are defined in~\eqref{vectors.v}, as required.
\end{proof}

\setcounter{equation}{0}
%%%%%%%%%%%%%%%%%%%%%%%%%%%%%%%%%%%%%%%%%%%%%%%%%%%%%%%%%%%%%%%%%%%%%%%%%%%%%%%%%%%%%%%%%%%%
\section{Tridiagonal matrices with alternating signs main diagonal}\label{section:alternate}
%%%%%%%%%%%%%%%%%%%%%%%%%%%%%%%%%%%%%%%%%%%%%%%%%%%%%%%%%%%%%%%%%%%%%%%%%%%%%%%%%%%%%%%%%%%%

Consider the matrix
\begin{equation}\label{Matrix.A}
A_n=J_n+xE_n,
\end{equation}
%
%where the matrix $E_n=\{e_{ij}\}_{i,j=1}^{n}$ is defined in~\eqref{Matrix.E}
%as follows
%
%\begin{equation*}%\label{Matrix.E}
%e_{ij}=
%\begin{cases}
%&(-1)^{i-1}\quad\text{if}\quad i=j,\\
%&\quad0\quad\qquad\text{if}\quad i\neq j,\\
%\end{cases}
%\end{equation*}
%
where the matrices $J_n$ and $E_n=\{e_{ij}\}_{i,j=1}^{n}$ are defined in~\eqref{main.matrix.intro} and~\eqref{Matrix.E}, respectively, so
the main diagonal of the matrix $A_n$ contains entries of the same non-zero absolute value and of alternating signs. The following fact holds.
\begin{theorem}\label{th:An-spectrum}
The spectrum of the matrix $A_n$ defined in~\eqref{Matrix.A} has the form
\begin{equation}\label{spectrum.A.1}
\sigma\left(A_{2l}\right)=\{\pm\sqrt{\lambda_1^2+x^2},\pm\sqrt{\lambda_2^2+x^2},\ldots,\pm\sqrt{\lambda_l^2+x^2}\},
\end{equation}
and
\begin{equation}\label{spectrum.A.2}
\sigma\left(A_{2l+1}\right)=\{x,\pm\sqrt{\lambda_1^2+x^2},\pm\sqrt{\lambda_2^2+x^2},\ldots,\pm\sqrt{\lambda_l^2+x^2}\}
\end{equation}
for~$l=\big\lfloor\frac n2\big\rfloor$, where $\{\lambda_i\}_{i=1}^l$ belong to the spectrum of $J_n$, see~\eqref{spectrum.J}.
\end{theorem}

In particular, if $x^2=-\lambda_j^2\ne 0$ for some~$j$, then~$0$ is an eigenvalue of $A_n$ of even multiplicity.
%
% By Theorem~\ref{th:J.spectrum}, none of~$\lambda_k$ may be zero for even~$n=2l$. 
Moreover, for~$n=2l+1$ some~$\lambda_k$ may vanish, in which case formul\ae~\eqref{spectrum.A.1}--\eqref{spectrum.A.2}
show the existence of eigenvalues~$\pm x$: the eigenvalue~$x$ of~$A_{2l+1}$ has odd multiplicity,
while~$-x$ is its eigenvalue of even multiplicity. 

\begin{proof}
Note that the easily verifiable identity $J_nE_n+E_nJ_n=0$ implies that
% the square of the matrix $A_n$ can be written as
%
\begin{equation}\label{A.square}
A^2_n=(J_n+xE_n)^2=J_n^2+xJ_nE_n+xE_nJ_n+x^2I_n=J_n^2+x^2I_n^2,
\end{equation}
where $I_n$ is the $n\times n$ identity matrix.
\medskip

Let $n=2l$. In this case, $\lambda_k\neq0$ for any $k=1,\ldots,l$. Suppose first that all the $\lambda_k$
are distinct. From~\eqref{A.square}, it follows that all the eigenvalues of~$A_{2l}^2$ are double and equal to~$\lambda_k^2+x^2$ for some~$k=1,\dots,l$. According to~\cite[Chapter VIII, \S 6--7]{GantmacherI}, for each~$k$ the numbers~$\sqrt{\lambda_k^2+x^2}$ and/or~$-\sqrt{\lambda_k^2+x^2}$ (and only they) are in the spectrum of the matrix~$A_{2l}$. Our aim is to show that the spectrum of~$A_{2l}$ is indeed given by~\eqref{spectrum.A.1}, where all eigenvalues are simple possibly excluding the double eigenvalue~$0$ corresponding to~$x^2=-\lambda_k^2$.
We do this by finding an explicit expression for the eigenvector of~$A_{2l}$ for each eigenvalue.

\vspace{1mm}

Let~$\mu=\sqrt{\lambda_k^2+x^2}\neq0$ for some $k$ and some fixed branch of the square root. If~$\mu$ is an eigenvalue of $A_{2l}$, then there exists a corresponding eigenvector~$\bm v_0^{(\mu)}$ satisfying
\begin{equation*}
A_{2l}^2\bm v_0^{(\mu)}=\mu^2\bm v_0^{(\mu)}=(\lambda_k^2+x^2)\,\bm v_0^{(\mu)}.
\end{equation*}
Then, due to
\begin{equation*}
A_{2l}^2 \bm u_0^{(\lambda_{k})} = (\lambda_k^2 + x^2) \bm u_0^{(\lambda_{k})}
\quad\text{and}\quad
A_{2l}^2 \bm u_0^{(-\lambda_{k})} = (\lambda_k^2 + x^2) \bm u_0^{(-\lambda_{k})}
,
\end{equation*}
where the eigenvectors~$\bm u_0^{(\lambda_{k})}$ and~$\bm u_0^{(-\lambda_{k})}$ of~$J_{2l}$ correspond to~$\lambda_{k}$ and~$-\lambda_k$ respectively,
we have
\begin{equation*}
\bm v_0^{(\mu)} = \alpha \bm u_0^{(\lambda_{k})} + \beta \bm u_0^{(-\lambda_{k})},
\end{equation*}
for a certain choice of the coefficients~$\alpha$ and~$\beta$. Therefore,
\begin{align*}
  \mu\big(\alpha \bm u_0^{(\lambda_{k})} + \beta\bm u_0^{(-\lambda_{k})}\big)
  &=
    \mu \bm v_0^{(\mu)}
    =
    A_{2l} \bm v_0^{(\mu)}
    =
    (J_{2l}+xE_{2l})\big(\alpha \bm u_0^{(\lambda_{k})} + \beta \bm u_0^{(-\lambda_{k})}\big)
  \\
  &=
    \alpha J_{2l} \bm u_0^{(\lambda_{k})} + \alpha x \bm E_{2l} u_0^{(\lambda_{k})} + \beta J_{2l} \bm u_0^{(-\lambda_{k})} + \beta x \bm E_{2l} u_0^{(-\lambda_{k})}
  \\
  &=
    (\alpha \lambda +\beta x) \bm u_0^{(\lambda_{k})} + (\alpha x - \beta\lambda_k)\bm u_0^{(-\lambda_{k})}
    ,
\end{align*}
%
% whence, due to the linear independence of eigenvectors, we have a homogeneous linear system with
% respect to~$\alpha,\beta$:
% %
% \begin{equation*}\label{eq:vmu_via_ulam_cf}
%     \left\{
%         \begin{matrix}
%             \left(\lambda_k-\mu\right)\alpha+ x\beta =0
%             \\
%             x \alpha-\left(\lambda_k+\mu\right)\beta =0
%         \end{matrix}
%     \right.
%     .
% \end{equation*}
% %
%The determinant of this system is zero,
whence on choosing~$\alpha=1$ we obtain~$\beta= \frac x{\lambda_k+\mu}=\frac{\mu-\lambda_k}x$.
So,
\begin{equation*}
\bm v_0^{(\mu)}= \bm u_0^{(\lambda_{k})} + \frac x{\lambda+\mu} \bm u_0^{(-\lambda_{k})}
,
\end{equation*}
and, on taking another branch of the square root ($-\mu$ instead of~$\mu$), it follows that~$\bm v_0^{(-\mu)}$ may be set to be
\begin{equation*}
\bm u_0^{(\lambda_{k})} + \frac x{\lambda_k-\mu} \bm u_0^{(-\lambda_{k})}
=\bm u_0^{(\lambda_{k})} - \frac {\lambda_k+\mu}x \bm u_0^{(-\lambda_{k})}.
\end{equation*}
The last expression for~$\bm v_0^{(-\mu)}$ degenerates as~$x\to 0$, so stretching it by the factor~$-\frac x{\lambda_k+\mu}$ leads to a more convenient renormalisation
\begin{equation*}
\bm v_0^{(-\mu)}= \bm u_0^{(-\lambda_{k})} - \frac{x}{\lambda+\mu} \bm u_0^{(\lambda_{k})}.
\end{equation*}
Thus, both numbers $\mu\neq0$ and $-\mu$ belong to the spectrum $A_{2l}$
provided $J_{2l}$ has only simple eigenvalues.

\vspace{1mm}

Suppose now that $x=\pm i\lambda_j$ for a fixed $j$, and $\bm u_0^{(\lambda_{j})}$ and $\bm u_0^{(-\lambda_{j})}$
are the eigenvectors of $J_{2l}$ corresponding to the eigenvalues $\lambda_j$ and $-\lambda_j$, respectively.
Then it is easy to see that
\begin{equation*}
A_{2l}(\bm u_0^{(\lambda_{j})}\pm i\bm u_0^{(-\lambda_{j})})=(J_{2l}\pm i\lambda_jE_{2l})(\bm u_0^{(\lambda_{j})}\pm i\bm u_0^{(-\lambda_{j})})=0
\quad\text{and}\quad
\dfrac1{2\lambda}A_{2l}(\bm u_0^{(\lambda_{j})}\mp i\bm u_0^{(-\lambda_{j})})=\bm u_0^{(\lambda_{j})}\pm i\bm u_0^{(-\lambda_{j})},
\end{equation*}
so
\begin{equation*}
\bm v_0^{(0)}=\bm u_0^{(\lambda_{j})}\pm i\bm u_0^{(-\lambda_{j})}\qquad\text{and}\qquad v_1^{(0)}=\dfrac{\bm u_0^{(\lambda_{j})}\mp i\bm u_0^{(-\lambda_{j})}}{2\lambda}
\end{equation*}
are, respectively, the eigenvector and generalised eigenvector of the zero eigenvalue of the matrix $A_{2l}$.%, provided $x=\pm i\lambda_j$.

Thus, if all eigenvalues $\pm\lambda_k$ of~$J_{2l}$ are simple, then $A_{2l}$ has simple non-zero eigenvalues of the form
$\pm\sqrt{\lambda_k^2+x^2}$ (and, in the case $x=\pm i\lambda_j$, a double zero eigenvalue). Now,
if some $\lambda_k$ is an eigenvalue
of~$J_{2l}$ of multiplicity~$r$ and $x^2\neq-\lambda_k^2$, then due to continuous dependence of the characteristic polynomial on its roots%
% \footnote{Clear that one can modify the coefficients of~$J_{2l}$ so that its spectrum becomes simple and distinct, see the Sylvester-Kac matrix~$K_{2l-1}$ above.}%
, the eigenvalues
$\pm\sqrt{\lambda_k^2+x^2}$ of $A_{2l}$ are also of multiplicity~$r$. Analogously, if $x^2=-\lambda_j^2$ and $\lambda_j$ is an eigenvalue
of $J_{2l}$ of multiplicity~$r$, then by continuity the eigenvalue $\mu=0$ of $A_{2l}$ is of multiplicity~$2r$.
\medskip

Suppose now that $n=2l+1$. From~\eqref{A.square} and~\cite[Chapter VIII, \S 6--7]{GantmacherI}, it follows
that the set
\begin{equation*}
\{\pm x,\pm\sqrt{\lambda_1^2+x^2},\pm\sqrt{\lambda_2^2+x^2},\ldots,\pm\sqrt{\lambda_l^2+x^2}\}.
\end{equation*}
contains all possible eigenvalues of the matrix $A_{2l+1}$.
Similarly to the case of $n=2l$, if non-zero eigenvalues $\pm\lambda_k$ of $J_{2l+1}$ have multiplicity~$r$, one can show that the matrix $A_{2l+1}$ has eigenvalues $\pm\sqrt{\lambda_k^2+x^2}$ of multiplicity~$r$ for~$x^2\neq-\lambda^2_k$, or the
zero eigenvalue of multiplicity~$2r$ for~$x^2=-\lambda^2_k$. (In fact, the relations between
the corresponding eigenvectors of~$A_{n}$ and~$J_{n}$ for~$n=2l+1$ remain the same as for~$n=2l$.)

\vspace{1mm}

Moreover, $\mu=x$ is always an eigenvalue of $A_{2l+1}$. Indeed, if $J_{2l+1}\bm u_0^{(0)}=0$, then~$E_{2l+1}\bm u_0^{(0)}=\bm u_0^{(0)}$ according to Remark~\ref{Matrix.E-lam0}, and hence
\begin{equation}\label{eq:eigenvector_lam=0}
    A_{2l+1} \bm u_0^{(0)}=J_{2l+1}\bm u_0^{(0)} + x E_{2l+1}\bm u_0^{(0)}=x \bm u_0^{(0)}.
\end{equation}
Suppose now that $J_{2l+1}$ has a zero eigenvalue of multiplicity $3$, and
\begin{equation*}
J_{2l+1}\bm u_0^{(0)}=0,\quad J_{2l+1}\bm u_1^{(0)}=\bm u_0^{(0)},\quad J_{2l+1}\bm u_2^{(0)}=\bm u_1^{(0)}.
\end{equation*}
In this case, $-x$ is also an eigenvalue of the matrix~$A_{2l+1}$ with the eigenvector~$\bm v_0^{(-x)}=\bm u_0^{(0)}-2x\bm u_1^{(0)}$. Indeed,
\begin{equation*}
A_{2l+1} (\bm u_0^{(0)}-2x\bm u^{(0)}_1)
=x\bm u_0^{(0)} -2x (J_{2l+1}\bm u^{(0)}_1 + x E_{2l+1}\bm u^{(0)}_1)
=x\bm u_0^{(0)} -2x (\bm u_0^{(0)} - x\bm u^{(0)}_1) =
-x (\bm u_0^{(0)}-2x\bm u^{(0)}_1).
\end{equation*}
Moreover, in this case $x$ is an eigenvalue of $A_{2l+1}$ of multiplicity at least $2$, and by~\eqref{vectors.v} the vector $\bm v_1^{(x)}=\bm u_1^{(0)}+2x\bm u_2^{(0)}$ satisfies
\begin{equation*}
(A_{2l+1}-xI_{2l+1})\bm v_1^{(x)}=(J_{2l+1}+xE_{2l+1}-xI_{2l+1})\bm v_1^{(x)}=\bm u_0^{(0)}+2x\bm u_1^{(0)}-x\bm u_1^{(0)}+2x^2\bm u_2^{(0)}-x\bm u_1^{(0)}-2x^2\bm u_2^{(0)}=\bm u_0^{(0)}.
\end{equation*}
Hence, if~$0$ is a triple eigenvalue of~$J_{2l+1}$, then $x$ is double and~$-x$ is simple due to multiplicities of the other eigenvalues of~$A_{2l+1}$.

Now by continuity we get that if~$0$ is an eigenvalue of~$J_{2l+1}$ of multiplicity $2r+1$, $r\geqslant0$, then $x$ is an eigenvalue of $A_{2l+1}$ of multiplicity~$r+1$ while~$-x$ is of multiplicity~$r$. Consequently, formul\ae~\eqref{spectrum.A.1}--\eqref{spectrum.A.2}
completely describe the spectrum of $A_n$.
\end{proof}

As a consequence of this theorem one gets the following formul\ae\ generalising the result of~\cite{Kilic_2013}.

\begin{corol}
For the matrix $A_n$ defined in~\eqref{Matrix.A},
\begin{equation*}%\label{eq:characteristic.last}
\det A_{2l}=(-1)^l\prod\limits_{k=1}^{l}\left(x^2+\lambda_k^2\right)
\quad\text{and}\quad
\det A_{2l+1}=(-1)^l x\prod\limits_{k=1}^{l}\left(x^2+\lambda_k^2\right),
\end{equation*}
where some of the numbers~$\lambda_k$ in the second product can be zero.
\end{corol}

%%%%%%%%%%%%%%%%%%%%%%%%%%%%%%%%%%%%%%%%%%%%%%%%%%%%%%%
\subsection{Eigenvectors and generalised eigenvectors}
%%%%%%%%%%%%%%%%%%%%%%%%%%%%%%%%%%%%%%%%%%%%%%%%%%%%%%%

\newcommand{\ulam}[1][]{\ensuremath{\bm u^{({#1}\lambda_k)}}}
\newcommand{\vmu}[1][]{\ensuremath{\bm v^{({#1}\mu)}}}

Let us list the explicit expressions for eigenvectors and first generalised eigenvectors for all possible eigenvalues of~$A_n$.

\vspace{2mm}

\noindent $1)$ $\mu=x$. According to Theorem~\ref{th:An-spectrum}, if $\lambda=0$ is an eigenvalue of $J_n$ of multiplicity $2r+1$, then
$\mu=x$ is an eigenvalue of~$A_n$ of multiplicity~$r+1$. In the proof of that theorem we showed that the eigenvector and first generalised eigenvector
of $A_n$ corresponding to $\mu=x$ are
\begin{equation*}
\begin{array}{l}
\bm v_0^{(x)}=\bm u_0^{(0)},\\
\bm v_1^{(x)}=\bm u_1^{(0)}+2x\bm u_2^{(0)},
\end{array}
\end{equation*}
where $\bm u_k^{(0)}$, $k=0,1,2$, are the eigenvector and generalised eigenvectors of $J_n$ corresponding to the eigenvalue~$\lambda=0$.

\vspace{1mm}

\noindent $2)$ $\mu=-x$. By Theorem~\ref{th:An-spectrum}, it is an eigenvalue of~$A_{n}$ only if~$\lambda=0$ is an eigenvalue
of~$J_n$ of multiplicity at least~$3$; if multiplicity of~$\lambda=0$ is at least $5$, then
$\mu=-x$ is a multiple eigenvalue of $A_n$. In the proof of Theorem~\ref{th:An-spectrum} we showed that the eigenvector
of $A_n$ corresponding to $\mu=-x$ has the form
\begin{equation*}
\bm v_0^{(-x)}=\bm u_0^{(0)}-2x\bm u_1^{(0)}.
\end{equation*}

Let us find $\bm v_1^{(-x)}$. Since by definition
\begin{equation}\label{gen.vect.-x}
(A_n+xI_n)\bm v_1^{(-x)}=\bm v_0^{(-x)},
\end{equation}
with use of~\eqref{A.square} one has
\begin{equation*}
0=(A_n+xI_n)^2\bm v_1^{(-x)}=(J_n^2+x^2I_n)\bm v_1^{(-x)}-2x^2\bm v_1^{(-x)}+2x\bm v_0^{(-x)}+x^2\bm v_1^{(-x)}=
J_n^2\bm v_1^{(-x)}+2x\bm v_0^{(-x)},
\end{equation*}
so that
\begin{equation*}
J_n^2\bm v_1^{(-x)}=-2x\bm v_0^{(-x)}=-2x\bm u_0^{(0)}+4x^2\bm u_1^{(0)}.
\end{equation*}
Therefore, the vector $\bm v_1^{(-x)}$ is a linear combination of the vectors $\bm u_k^{(0)}$, $k=0,\ldots,3$,
\begin{equation*}
\bm v_1^{(-x)}=\alpha\bm u_0^{(0)}+\beta\bm u_1^{(0)}-2x\bm u_2^{(0)}+4x^2\bm u_3^{(0)}.
\end{equation*}
Substituting this into~\eqref{gen.vect.-x} (with $\alpha=0$) gives us
\begin{equation*}
\bm v_1^{(-x)}=\bm u_1^{(0)}-2x\bm u_2^{(0)}+4x^2\bm u_3^{(0)}.
\end{equation*}

\vspace{1mm}

\noindent $3)$ $\mu=0$. If $\lambda_j\neq0$ and $x=\pm i\lambda_j$, then according to the proof of Theorem~\ref{th:An-spectrum}, the correspondent eigenvector and first generalised eigenvectors of $A_n$ are
\begin{equation}\label{eigenvect.mu=0}
%\begin{array}{l}
\bm v_0^{(0)}=\bm u_0^{(\lambda_j)}\pm i\bm u_0^{(-\lambda_j)}
\quad\text{and}\quad
\bm v_1^{(0)}=\dfrac{\bm u_0^{(\lambda_j)}\mp  i\bm u_0^{(-\lambda_j)}}{2\lambda_j}.
%\end{array}
\end{equation}

\vspace{1mm}

\noindent $4)$ $\mu=\sqrt{x^2+\lambda_k^2}$ for some eigenvalue $\lambda_k\neq0$ of the matrix $J_n$, where we choose any fixed branch of the complex square root. From the proof of Theorem~\ref{th:An-spectrum}, we know that
\begin{equation*}
\bm v_0^{(\mu)}= \bm u_0^{(\lambda_{k})} + \frac x{\lambda_k+\mu} \bm u_0^{(-\lambda_{k})}
\end{equation*}

If $\mu$ is a multiple eigenvalue, the same approach we used to find the eigenvector allows to express the first generalised eigenvector of~$A_n$ corresponding to~$\mu$ as
combinations of the eigenvectors and generalised eigenvectors of~$J_n$ corresponding
to~$\lambda_k$. Namely, let
\[J_n \bm u_1^{(\lambda_k)} = \lambda \bm u_1^{(\lambda_k)} + \bm u_0^{(\lambda_k)}.%\quad\text{for}\quad \lambda_k\neq 0.
\]
If~$\vmu_1$ satisfying~$A_n \vmu_1 = \mu \vmu_1 + \vmu_0$ is sought in the form
\begin{equation*}%\label{eq:vmu_gen_via_ulam}
    \vmu_1= \alpha \bm u_1^{(\lambda_k)} + \beta \bm u_1^{(-\lambda_k)} + \gamma \bm u_0^{(\lambda_k)} + \delta \bm u_0^{(-\lambda_k)},
\end{equation*}
then the same approach as above yields
\begin{equation}\label{eq:vmu_gen_via_ulam}
\vmu_1
=\dfrac1{2\lambda_k}\left(\bm u_0^{(\lambda_k)}-\dfrac{x}{\lambda_k+\mu}\bm u_0^{(-\lambda_k)}\right)
+
\dfrac{\mu}{\lambda_k}
\left(\bm u_1^{(\lambda_k)}
-\dfrac{x}{\lambda_k+\mu}\bm u_1^{(-\lambda_k)}\right).
\end{equation}
Here the chosen values of~$\alpha,\beta,\gamma,\delta$ are natural in the sense that the $\bm v_0^{(\mu)}$
and $\bm v_1^{(\mu)}$ become~\eqref{eigenvect.mu=0} as $\mu\to0$ and do not degenerate as $x\to0$.

Analogously, one gets
\begin{equation*}
\begin{array}{l}
\displaystyle\bm v_0^{(-\mu)}= \bm u_0^{(-\lambda_{k})} - \frac{x}{\lambda_k+\mu} \bm u_0^{(\lambda_{k})}
\qquad\text{and}\\[8pt]
\bm v_1^{(-\mu)}=-\dfrac1{2\lambda_k}\left(\dfrac{x}{\lambda_k+\mu}\bm u_0^{(\lambda_k)}+\bm u_0^{(-\lambda_k)}\right)
+
\dfrac{\mu}{\lambda_k}
\left(\dfrac{x}{\lambda_k+\mu}\bm u_1^{(\lambda_k)}+\bm u_1^{(-\lambda_k)}\right).

\end{array}
\end{equation*}

We note the choice of generalised eigenvectors is non-unique%
\footnote{For instance, for any constant~$\varepsilon\in\mathbb C$ the
    combination~$\ulam_1+\varepsilon\ulam_0$ is a generalised eigenvector of~$J_n$ corresponding
    to~$\lambda$.}%
, and the expression~\eqref{eq:vmu_gen_via_ulam} may be replaced with a different (and in a
sense more general) formula considered in our forthcoming publication~\cite{DyTy_2}. That
publication also gives a detailed description of the generalised eigenvectors corresponding to
the eigenvalues~$\pm x$ of~$A_{2l+1}$ induced by the non-simple eigenvalue~$\lambda=0$
of~$J_{2l+1}$.

\setcounter{equation}{0}
%%%%%%%%%%%%%%%%%%%%%%%%%%%%%%%%%%%%%%%%%%%%%%%%%%%%%%%%%%%%%%%%%%%%%%%%%%%%%%%%%%%%%%%%%%
\section{Tridiagonal matrices with two-periodic main diagonal}\label{section:two.periodic}
%%%%%%%%%%%%%%%%%%%%%%%%%%%%%%%%%%%%%%%%%%%%%%%%%%%%%%%%%%%%%%%%%%%%%%%%%%%%%%%%%%%%%%%%%%

Consider now the matrix
\begin{equation*}\label{Kilic.Arikan.matrix}
B_n=
\begin{pmatrix}
    b_1 & c_1 &  0 &\dots&   0   & 0 \\
    a_1 & b_2 &c_2 &\dots&   0   & 0 \\
     0  &a_2 & b_3 &\dots&   0   & 0 \\
    \vdots&\vdots&\vdots&\ddots&\vdots&\vdots\\
     0  &  0  &  0  &\dots&b_{n-1}& c_{n-1}\\
     0  &  0  &  0  &\dots&a_{n-1}&b_{n}\\
\end{pmatrix},
\quad a_k,c_k\in\mathbb{C}\setminus\{0\},\quad
b_k=\begin{cases}
x&\text{if}\quad k\ \ \text{is odd},\\
y&\text{if}\quad k\ \ \text{is even}.\\
\end{cases}
\end{equation*}

It is easy to see that
\begin{equation*}
B_{n}=J_n+\dfrac{x-y}{2}E_n+\dfrac{x+y}{2}I_n,
\end{equation*}
where $J_n$ is defined in~\eqref{main.matrix.intro}, so from~\eqref{spectrum.A.1}--\eqref{spectrum.A.2} we obtain
\begin{equation}\label{spectrum.B.1}
\sigma\left(B_{2l}\right)=\left\{\dfrac{x+y}2\pm\dfrac12\sqrt{4\lambda_1^2+(x-y)^2},\dfrac{x+y}2\pm\dfrac12\sqrt{4\lambda_2^2+(x-y)^2},\ldots,\dfrac{x+y}2\pm\dfrac12\sqrt{4\lambda_l^2+(x-y)^2}\right\},
\end{equation}
and
\begin{equation}\label{spectrum.B.2}
\sigma\left(B_{2l+1}\right)=\left\{x,\dfrac{x+y}2\pm\dfrac12\sqrt{4\lambda_1^2+(x-y)^2},\dfrac{x+y}2\pm\dfrac12\sqrt{4\lambda_2^2+(x-y)^2},
\ldots,\dfrac{x+y}2\pm\dfrac12\sqrt{4\lambda_l^2+(x-y)^2}\right\}.
\end{equation}
These formul\ae\ can be obtained from~\eqref{spectrum.A.1}--\eqref{spectrum.A.2} by replacing $x$ with $\dfrac{x-y}2$ and then by adding $\dfrac{x+y}2$ to all the eigenvalues.

From~\eqref{spectrum.B.1}--\eqref{spectrum.B.2}, one can easily obtain that the determinant of $B_n$ has the form
\begin{equation}\label{det.B}
\det B_{2l}=\prod\limits_{k=1}^{l}\left(xy-\lambda_k^2\right),\qquad\quad \det B_{2l+1}=x\prod\limits_{k=1}^{l}\left(xy-\lambda_k^2\right).
\end{equation}
On letting~$J_n$ to be the Sylvester-Kac matrix or its main principal submatrix, the formul\ae~\eqref{spectrum.B.1}--\eqref{det.B} generalise the results of the works~\cite{Fonseca_2018,KilicArikan_2016}, as well as an analogous transition in~\cite{FonsecaKilic_2019}.

Observe that the eigenvectors and generalised eigenvalues of the matrices~$A_n$ and~$B_n$ are
related through replacing~$x$ to~$\frac{x-y}2$.

The right eigenvalues of~$A_n$ and~$B_n$ may
be obtained using the following remark.

\begin{remark}
Note that if the (right) eigenvectors and generalised eigenvectors of some irreducible tridiagonal matrix are known, then
it is easy to find its left eigenvectors: that is, the eigenvectors of a matrix
\[
\mathcal J_n
=
\begin{pmatrix}
    b_1 & c_1 &  0 &\dots&   0   & 0 \\
    a_1 & b_2 &c_2 &\dots&   0   & 0 \\
     0  &a_2 & b_3 &\dots&   0   & 0 \\
    \vdots&\vdots&\vdots&\ddots&\vdots&\vdots\\
     0  &  0  &  0  &\dots&b_{n-1}& c_{n-1}\\
     0  &  0  &  0  &\dots&a_{n-1}&b_{n}\\
\end{pmatrix},
\quad a_k,c_k\in\mathbb{C}\setminus\{0\}
\]
are related to the eigenvectors of its transposed $\mathcal J_n^{T}$.
It is clear that the spectra of $\mathcal J_n^{T}$ and $\mathcal J_n$ coincide. So if  $\widetilde{\textit{\textbf{u}}}_0^{(\lambda)}$
is the eigenvector of $\mathcal J_n^{T}$ corresponding to the eigenvalue $\lambda$, then the obvious formula
\begin{equation}\label{Transpose.repres}
\mathcal J_n^{T}=D_n^{-1}\mathcal J_nD_n,
\end{equation}
where the diagonal matrix $D$ has the form
\begin{equation*}
D=
\begin{pmatrix}
    d_1 & 0 &  0 &\dots&   0   & 0 \\
    0 & d_2 &0 &\dots&   0   & 0 \\
     0  &0 & d_3 &\dots&   0   & 0 \\
    \vdots&\vdots&\vdots&\ddots&\vdots&\vdots\\
     0  &  0  &  0  &\dots&d_{n-1}& 0\\
     0  &  0  &  0  &\dots&0&d_{n}\\
\end{pmatrix},\qquad\text{with}\quad d_1=1,\ \ d_k=\dfrac{a_1a_2\cdots a_k}{c_1c_2\cdots c_k}, \quad k=1,\ldots,n-1.
\end{equation*}
implies by induction that
\begin{equation*}%\label{Transpose.repres}
\widetilde{\textit{\textbf{u}}}_0^{(\lambda)}=D_n^{-1}\textit{\textbf{u}}_0^{(\lambda)},
\qquad
\widetilde{\textit{\textbf{u}}}_j^{(\lambda)}=D_n^{-1}\textit{\textbf{u}}_j^{(\lambda)},\ \ \ j=1,\ldots,k-1,
\end{equation*}
where $k$ is the multiplicity of the eigenvalue $\lambda$.
\end{remark}

\setcounter{equation}{0}
%%%%%%%%%%%%%%%%%%%%%%%%%%%%%%%%%%%%%%%%%%%%%%%%%%%%%%%%%%%%%%%%%%%%%%
\section*{Acknowledgement}
%%%%%%%%%%%%%%%%%%%%%%%%%%%%%%%%%%%%%%%%%%%%%%%%%%%%%%%%%%%%%%%%%%%%%%

M.\,Tyaglov was partially supported by the  National Natural Science Foundation of China, grant no.~11871336.

\end{document}